\documentclass[11pt,reqno]{amsart}
\usepackage{amssymb,mathtools,calc,verbatim,enumitem,tikz,url,hyperref,mathrsfs,cite,fullpage}
\usepackage{bbm}
\usepackage{stmaryrd}
\usepackage{textcomp}
\usepackage{setspace}
\usepackage{amsthm}
\usepackage{amsmath}
\usepackage{graphicx}
\usepackage{marvosym}
\usepackage{empheq}
\usepackage{latexsym}
\usepackage{fontenc}
\usepackage{color}

\usepackage{cleveref}
\usepackage{dsfont}

\usepackage{todonotes}

\newtheorem{theorem}{Theorem}[section]
\newtheorem{lemma}[theorem]{Lemma}

\newtheorem{proposition}[theorem]{Proposition}

\newtheorem*{claim*}{Claim}

\theoremstyle{definition}
\newtheorem{definition}[theorem]{Definition}
\newtheorem{question}[theorem]{Question}
\newtheorem*{qu*}{Question}
\theoremstyle{remark}
\newtheorem{remark}[theorem]{Remark}

\newcommand\N{\mathbb{N}}
\newcommand\E{\operatorname{\mathbb{E}}}

\renewcommand\Pr{\operatorname{\mathbb{P}}}

\renewcommand\leq{\leqslant}
\renewcommand\geq{\geqslant}
\renewcommand\le{\leqslant}
\renewcommand\ge{\geqslant}
\renewcommand\to{\rightarrow}

\def\R{\mathbb{R}}

\def\<{\langle }
\def\>{\rangle }

\newcommand{\wt}{\operatorname{wt}}

\begin{document}

\title{The sharp threshold for rainbow stackings of random edge-colourings}
\author{
Hong Liu
\and
Guorui Ma
\and 
Yangrui Xiang
\and
Zhifei Yan}

\address{ECOPRO, Institute for Basic Science, 55 Expo-ro, Yuseong-gu, Daejeon, 34126, Korea}\email{\{hongliu,zhifeiyan\}@ibs.re.kr}

\address{Shanghai Institute for Mathematics and Interdisciplinary Sciences (SIMIS), Shanghai, 200433, China; Research institute of Intelligent Complex Systems, Fudan University, Shanghai, 200433, China}\email{mgr18@tsinghua.org.cn}

\address{Department of Mathematics, Louisiana State University, 303 Lockett Hall, Baton Rouge, Louisiana, 70801, USA}\email{yxiang8@lsu.edu}

\thanks{}

\begin{abstract}
A rainbow stacking of $m$ independent, uniformly random $r$-edge-colourings of $K_n$ is a tuple of vertex permutations that superimposes the colourings such that no two edges of the same colour overlap. The study of the critical palette size $r$ required for the existence of such stackings was recently initiated by Alon, Defant, and Kravitz [Bull. Lond. Math. Soc., 57, 2025], who bounded the phase transition within a constant-order window around $\frac{m\binom{n}{2}}{2\log(n!)}$.

We determine the constant term in this transition.  For every fixed $m\ge2$ and every function $\omega(n)\to\infty$, with high probability there is no rainbow stacking if
\[
        r\le \frac{m\binom{n}{2}}{2\log(n!)}+\frac{2m-1}{6}-\frac{\omega(n)}{(\log n)^2},
\]
while with high probability there is one if
\[
        r\ge \frac{m\binom{n}{2}}{2\log(n!)}+\frac{2m-1}{6}+\frac{\omega(n)}{(\log n)^2}.
\]
Our proof combines a chromatic-polynomial expansion for an auxiliary conflict graph with a refined estimate of the associated weighted permutation sum.  Our result yields the exact threshold $\Big\lceil \frac{m\binom{n}{2}}{2\log(n!)}+\frac{2m-1}{6}\Big\rceil$ for a density-one set of integers $n$, resolving a problem of Alon, Defant and Kravitz.
\end{abstract}
	
\maketitle

\section{Introduction}
Let $\mathfrak{S}_{n}$ denote the symmetric group of permutations of the set $[n] := \{1, \dots, n\}$, and let $K_{n}$ denote the complete graph with vertex set $[n]$ and edge set $\binom{[n]}{2}$. For a palette $C_r$ of $r$ colours, let $\chi_{1}, \dots, \chi_{m} : \binom{[n]}{2} \to C_r$ be edge-colourings of $K_{n}$. A \emph{rainbow stacking} of $\chi_{1}, \dots, \chi_{m}$ is a tuple $\sigma = (\sigma_{1}, \dots, \sigma_{m}) \in \mathfrak{S}_{n}^{m}$ such that for each edge position $e \in \binom{[n]}{2}$, the colours
\[ \chi_{1}(\sigma_{1}^{-1}(e)), \dots, \chi_{m}(\sigma_{m}^{-1}(e)) \]
are all distinct. In other words, a rainbow stacking superimposes copies of $K_n$ such that no edge is stacked above another edge of the same colour.

The study of rainbow stackings was recently initiated by Alon, Defant, and Kravitz \cite{ADK}, who investigated the critical palette size $r$ required to guarantee the existence of a rainbow stacking when $\chi_{1}, \dots, \chi_{m}$ are independent uniformly random edge-colourings. Fixing the number of layers $m \ge 2$ and letting $n \to \infty$, they established that the phase transition lies within a window around the generic entropy bound $\frac{m\binom{n}{2}}{2\log(n!)}$.
Specifically, they proved that for any function $\omega(n) \to \infty$, a rainbow stacking does not exist with high probability if 
$$r \le \frac{m\binom{n}{2}}{2\log(n!)} - \frac{\omega(n)}{(\log n)^{2}}.$$
Conversely, applying the second moment method, they established existence with high probability provided
\[ r \ge \frac{m\binom{n}{2}}{2\log(n!)} + \frac{2m-1}{3} + \frac{m}{2\log n} + \frac{\omega(n)}{(\log n)^{2}}. \]

Previous bounds left a gap containing a constant term of $(2m-1)/3$ and an $O(1/\log n)$ term. The first issue is due to the difficulty of capturing dependencies among local colour conflicts when using general entropy bounds. The second issue is tied to bounding the second moment over highly dependent overlaps. To address these issues, we introduce modifications to both steps, which allow us to determine the sharp threshold.

Our first contribution handles the local obstruction. For a relative permutation tuple $\pi$, we encode the simultaneous rainbow conditions for the identity stacking and the $\pi$-stacking as a proper-colouring problem on an auxiliary bounded-degree graph $G_\pi$.  We can then apply a cluster expansion (Fadnavis's second-order asymptotic expansion \cite{Fadnavis}) for its chromatic polynomial. Its edge term gives the first-order collision contribution, while its triangle term records exactly the local correlations among overlapping constraints. A sharp triangle count in $G_\pi$ yields the correction $\exp\left(\frac{\wt(\pi)}{r-(2m-1)/6}\right)$,
which is the essential reason behind the constant $(2m-1)/6$ (see Section~\ref{sec:bound_chromatic}).

Our second contribution handles the global summation.  We decouple the $\binom m2$ pairwise relative-permutation overlaps by extracting a Kruskal maximum spanning tree on the layer set.  This reduces the variance estimate to $d=m-1$ ordered tree weights, but the near-identity endpoint remains the delicate part.  We keep the finer defect data (the number of moved vertices and the number of transpositions) rather than summing only over fixed two-sets.  In the near-identity regime, the Kruskal ordering forces a monotone ordering of the moved-vertex defects; the resulting positive suffix sums eliminate the previous $O(1/\log n)$ loss (see Section~\ref{sec:weighted_permutation_sum}). 

Combining these estimates gives the following threshold localization.

\begin{theorem}[Sharp threshold for rainbow stacking] \label{thm:main_stacking}
Fix an integer $m \ge 2$, and let $\omega:\N \to \R$ be any function satisfying $\omega(n)\to\infty$. For each $n$, let
$\chi_{1},\dots,\chi_{m} : \binom{[n]}{2} \to C_r$
be independent uniformly random $r$-edge-colourings of $K_{n}$.
\begin{enumerate}
    \item[\rm (a)] If 
    \[ r \le \frac{m\binom{n}{2}}{2\log(n!)} + \frac{2m-1}{6} - \frac{\omega(n)}{(\log n)^{2}}, \]
    then whp there is no rainbow stacking of $\chi_{1},\dots,\chi_{m}$.
    
    \item[\rm (b)] If 
    \[ r \ge \frac{m\binom{n}{2}}{2\log(n!)} + \frac{2m-1}{6} + \frac{\omega(n)}{(\log n)^{2}}, \]
    then whp there is a rainbow stacking of $\chi_{1},\dots,\chi_{m}$.  
\end{enumerate}
\end{theorem}

As $n\to\infty$, the width of the unresolved window tends to $0$. Consequently, for all sufficiently large $n$, there is at most one integer lying between the non-existence and existence bounds. Thus Theorem~\ref{thm:main_stacking} determines the exact integer threshold whenever
\[
r_0(n):=\frac{m\binom{n}{2}}{2\log(n!)}+\frac{2m-1}{6}
\]
is not within $\omega(n)/(\log n)^2$ of an integer. In particular, for a density-one set of integers $n$ (and hence infinitely many $n$), the threshold is exactly $\lceil r_0(n)\rceil$, completely resolving the exact threshold question posed in~\cite{ADK}.

The remainder of this paper is organized as follows. In Section~\ref{sec:first_moment}, we present the first-moment obstruction and establish the non-existence threshold (Theorem~\ref{thm:main_stacking}(a)). In Section~\ref{sec:second_moment}, we set up the second moment method by defining the auxiliary graph $G_\pi$ and formulating the two core technical bounds (Propositions~\ref{prop:N_pi_final_bound} and \ref{prop:weighted_permutation_sum}). This allows us to swiftly deduce the existence threshold (Theorem~\ref{thm:main_stacking}(b)). The proof of the first core bound, evaluating the chromatic polynomial via cluster expansion, is carried out in Section~\ref{sec:bound_chromatic}. In Section~\ref{sec:weighted_permutation_sum}, we prove the second core bound by estimating the highly dependent weighted permutation sum. Finally, concluding remarks are discussed in Section~\ref{sec:conclusion}.

\section{First moment and non-existence}\label{sec:first_moment}

In this section, we apply the first moment method to prove the non-existence threshold of Theorem~\ref{thm:main_stacking}(a). For convenience, throughout the paper we define the following parameters:
\[ N = \binom{n}{2}, \quad L = \log(n!), \quad R_{n} = \frac{mN}{2L}, \quad \text{and} \quad c_{m} = \frac{2m-1}{6}. \]

For each tuple $\sigma = (\sigma_1, \dots, \sigma_m) \in S_n^m$, let $Z_\sigma$ denote the indicator random variable of the event that $\sigma$ forms a rainbow stacking. Let 
$$Z = \sum_{\sigma \in S_n^m} Z_\sigma$$
be the random variable counting the total number of rainbow stackings. For a fixed edge position $e$ and a fixed permutation tuple $\sigma$, the random colours $\chi_1(\sigma_1^{-1}(e)), \dots, \chi_m(\sigma_m^{-1}(e))$ are independent and uniformly distributed in $C_r$. The probability that these $m$ colours are pairwise distinct is:
\[ \frac{r(r-1)\dots(r-m+1)}{r^m} = \prod_{i=1}^{m-1} \left(1 - \frac{i}{r}\right). \]
Since the colour assignments on different underlying edges are completely independent, the conditions for different final edge positions $e \in \binom{[n]}{2}$ are mutually independent. Therefore, for any fixed $\sigma \in S_n^m$, we have:
\begin{equation} \label{eq:EZ_sigma}
\E Z_\sigma = E_{n,m,r} := \prod_{i=1}^{m-1} \left(1 - \frac{i}{r}\right)^N.
\end{equation}
Summing over all $(n!)^m$ possible choices for $\sigma$, the expected number of rainbow stackings is:
\begin{equation} \label{eq:EZ_total}
\E Z = (n!)^m E_{n,m,r}.
\end{equation}

Notice there exists an inherent diagonal symmetry acting on the space of stackings. For any common re-labelling permutation $\tau \in S_n$, we define the shifted tuple as $\tau\sigma = (\tau\sigma_1, \dots, \tau\sigma_m)$. We first verify that this diagonal action is free. If $\tau\sigma_i = \tau'\sigma_i$ for every $i \in [m]$, then multiplying on the right by $\sigma_i^{-1}$ yields $\tau = \tau'$. Thus, every diagonal orbit has exactly $n!$ distinct tuples.

Furthermore, the rainbow property is invariant under this action. Indeed, for any edge $e$, the colour assigned by the $i$-th layer in the shifted stacking is $\chi_i((\tau\sigma_i)^{-1}(e)) = \chi_i(\sigma_i^{-1}(\tau^{-1}(e))).$
As $e$ ranges over $\binom{[n]}{2}$, so does the mapped edge $\tau^{-1}(e)$. Hence, the set of superimposed colours across the entire graph remains unchanged. Therefore, $\tau\sigma$ is a rainbow stacking if and only if $\sigma$ is.

Since every rainbow stacking belongs to a full diagonal orbit of size $n!$, the random variable $Z$ is always a multiple of $n!$. Therefore, if at least one rainbow stacking exists, its entire diagonal orbit of size $n!$ exists, which implies that $Z > 0 \implies Z \ge n!$. This rigidity allows us to apply Markov's inequality in the following sharp form:
\begin{equation} \label{eq:markov_sharp}
\Pr[Z > 0] = \Pr[Z \ge n!] \le \frac{\E Z}{n!}.
\end{equation}

We now formally establish the non-existence.

\begin{proof}[Proof of Theorem~\ref{thm:main_stacking}\textup{(a)}]
Let us replace $\omega(n)$ with the truncated function $\omega_0(n) = \min\{\omega(n), \log n\}$. Clearly, $\omega_0(n) \to \infty$ but $\omega_0(n) = O(\log n)$, which justifies the uniform validity of our subsequent Taylor expansions. Since increasing the palette size $r$ can only increase the probability of a configuration being rainbow, the upper bound case 
\[ r = R_n + c_m - \eta_n, \quad \text{where } \eta_n = \frac{\omega_0(n)}{(\log n)^2}, \]
is the hardest to prove. Taking the logarithm of \eqref{eq:EZ_total} scaled by $n!$, we have:
\begin{equation} \label{eq:log_first_moment}
\log \frac{\E Z}{n!} = (m-1)L + N \sum_{i=1}^{m-1} \log\left(1 - \frac{i}{r}\right).
\end{equation}
Since $r = \Theta\left(\frac{n}{\log n}\right)$, we have $N/r^3 = O((\log n)^3 / n) = o(1)$. Applying the Taylor expansion $\log(1-x) = -x - \frac{x^2}{2} + O(x^3)$ to \eqref{eq:log_first_moment} yields:
\begin{align}
\log \frac{\E Z}{n!} &= (m-1)L - \frac{N}{r}\sum_{i=1}^{m-1} i - \frac{N}{2r^2}\sum_{i=1}^{m-1} i^2 + O_m\left(\frac{N}{r^3}\right) \notag \\
&= (m-1)\left(L - \frac{mN}{2r} - \frac{m(2m-1)N}{12r^2}\right) + o(1). \label{eq:taylor_expanded}
\end{align}
We now analyze the expression inside the parentheses at $r = R_n + d$, where $d = c_m - \eta_n = O(1)$. By the definition of $R_n$, we have $\frac{mN}{2R_n} = L$. Uniformly for any bounded $d$, we expand the first-order term:
\begin{equation} \label{eq:exp_part1}
\frac{mN}{2(R_n + d)} = L \left(1 + \frac{d}{R_n}\right)^{-1} = L - \frac{Ld}{R_n} + o(1).
\end{equation}
Similarly, for the second-order term, we obtain:
\begin{equation} \label{eq:exp_part2}
\frac{m(2m-1)N}{12(R_n+d)^2} = \frac{m(2m-1)N}{12R_n^2} + o(1) = \frac{c_m L}{R_n} + o(1),
\end{equation}
where the identity $\frac{m(2m-1)N}{12R_n^2} = \frac{c_m L}{R_n}$ can be checked directly by substituting $R_n = \frac{mN}{2L}$. 

Inserting \eqref{eq:exp_part1} and \eqref{eq:exp_part2} into \eqref{eq:taylor_expanded} with $d = c_m - \eta_n$, the constant terms $c_m$ cancel precisely:
\begin{align*}
L - \frac{mN}{2r} - \frac{m(2m-1)N}{12r^2} &= L - \left(L - \frac{L(c_m - \eta_n)}{R_n}\right) - \frac{c_m L}{R_n} + o(1)= -\frac{\eta_n L}{R_n} + o(1).
\end{align*}
Finally, we compute the ratio $\frac{L}{R_n} = \frac{2L^2}{mN}$. Note that $L \sim n\log n$ and $N \sim n^2/2$, we have
\[ \frac{\eta_n L}{R_n} = \left(\frac{4}{m} + o(1)\right) \omega_0(n) \to \infty, \]
since $\frac{L}{R_n} = \left(\frac{4}{m} + o(1)\right) (\log n)^2$. Thus, $\log \frac{\E Z}{n!} \to -\infty$, which means $\E Z / n! \to 0$. By \eqref{eq:markov_sharp}, we have $\Pr[Z > 0] \to 0$ as $n \to \infty$.
\end{proof}

\medskip

\section{Second moment method and the existence threshold}\label{sec:second_moment}

To establish the existence threshold, we employ the second moment method. By vertex-relabeling invariance, the joint expectation $\E[Z_\sigma Z_\tau]$ depends only on the relative permutation tuple $\pi = (\sigma_1\tau_1^{-1},\ldots,\sigma_m\tau_m^{-1})$. Fixing the second stacking to the identity $\mathrm{id} = (\mathrm{id},\ldots,\mathrm{id})$, we can expand the second moment as:
\begin{equation}\label{eq:second_moment_reduction}
        \E Z^2=(n!)^m\sum_{\pi\in\mathfrak S_n^m} \E[Z_{\mathrm{id}}Z_\pi].
\end{equation}

To compute $\E[Z_{\mathrm{id}}Z_\pi]$, we encode the simultaneous rainbow condition of $\mathrm{id}$ and $\pi$ as a proper vertex-colouring problem on an auxiliary graph $G_\pi$.

\begin{definition}[Auxiliary Graph $G_\pi$]\label{def:aux_graph}
For a fixed $\pi\in\mathfrak S_n^m$, the graph $G_\pi$ is defined as follows:
\begin{itemize}
    \item \textbf{Vertices:} Let $\beta_k(e)$ represents the edge $e$ within the $k$-th colouring layer.
    $$V(G_\pi)=\left\{\beta_k(e): k\in[m],\ e\in\binom{[n]}2\right\}$$ 
    Note the total number of vertices is $mN$.
    
    \item \textbf{Edges:} There are no edges within the same layer. For distinct layers $k\neq \ell$, vertices $\beta_k(e)$ and $\beta_\ell(e')$ are adjacent if and only if they map to the same position in either the $\mathrm{id}$-stacking or the $\pi$-stacking. That is, if:
    \begin{enumerate}[label=\textup{(\roman*)}]
        \item $e=e'$ \quad (colour conflict in the $\mathrm{id}$-stacking), or
        \item $\pi_k(e)=\pi_\ell(e')$ \quad (colour conflict in the $\pi$-stacking).
    \end{enumerate}
\end{itemize}
\end{definition}

Assigning each vertex $\beta_k(e)$ its independent random colour $\chi_k(e)$ yields a uniformly random $r$-colouring of $G_\pi$. The stackings $\mathrm{id}$ and $\pi$ are simultaneously rainbow if and only if this colouring is proper. Counting the valid colourings via the chromatic polynomial $P_{G_\pi}(r)$, we have
\begin{equation}\label{eq:chromatic_prob_link}
\E[Z_{\mathrm{id}}Z_\pi] =r^{-mN}P_{G_\pi}(r).
\end{equation}

To capture the overlaps between permutations, we define the pairwise and total collision weights. Let $\tau_{k\ell}=\pi_k^{-1}\pi_\ell\in\mathfrak S_n$. Let $f_{k\ell}(\pi)$ and $t_{k\ell}(\pi)$ denote the number of fixed points and transpositions of $\tau_{k\ell}$, respectively. We define:
\begin{equation}\label{eq:weight_def}
        \wt_{k\ell}(\pi)=\binom{f_{k\ell}(\pi)}2+t_{k\ell}(\pi),
        \qquad
        \wt(\pi)=\sum_{1\leq k<\ell\leq m}\wt_{k\ell}(\pi).
\end{equation}
Notice that $\wt_{k\ell}(\pi)$ is the number of edges $e\in\binom{[n]}2$ fixed set-wise by the vertex permutation $\tau_{k\ell}$.  Such an edge is either formed by two fixed vertices of $\tau_{k\ell}$ or is the support of a transposition of $\tau_{k\ell}$.

To control the probability in \eqref{eq:chromatic_prob_link}, we establish the following tight upper bound on the chromatic polynomial. This strictly strengthens Lemma 2.1 in \cite{ADK}  by improving the exponent, and its proof is deferred to Section~\ref{sec:bound_chromatic}.

\begin{proposition}\label{prop:N_pi_final_bound}
Assume $N/r^3=o(1)$.  Uniformly for all $\pi\in\mathfrak S_n^m$,
\begin{equation}\label{eq:N_pi_formula}
        P_{G_\pi}(r) \leq (1+o(1))r^{mN}E_{n,m,r}^2 \exp\left(\frac{\wt(\pi)}{r-c_m}\right).
\end{equation}
\end{proposition}

Based on this bound, we must show that the exponential collision weights do not asymptotically inflate the second-moment summation. This requires a delicate structural decoupling which we defer to Section~\ref{sec:weighted_permutation_sum}.

\begin{proposition}\label{prop:weighted_permutation_sum}
Let $m\ge2$ be fixed and let $\omega(n)\to\infty$.  If $\rho\ge R_n+\frac{\omega(n)}{(\log n)^2}$, then
\begin{equation}\label{eq:target_sum}
       \sum_{\pi\in\mathfrak S_n^m}
        \exp\left(\frac{\wt(\pi)}\rho\right)
        =(n!)^m\cdot(1+o(1)).
\end{equation}
\end{proposition}

Note that Proposition 2.2 in \cite{ADK} establishes the same estimate as Proposition~\ref{prop:weighted_permutation_sum}, except that their assumption on $\rho$ is restricted to a narrower range:
\[
\rho\ge R_n+\frac{m}{\log n}+\frac{\omega(n)}{(\log n)^2}.
\]

Assuming the bounded chromatic estimate (Proposition~\ref{prop:N_pi_final_bound}) and the decoupled weighted permutation sum (Proposition~\ref{prop:weighted_permutation_sum}), the proof of the existence threshold is immediate.

\begin{proof}[Proof of Theorem~\ref{thm:main_stacking}\textup{(b)}]
Assume $r\geq R_n+c_m+\frac{\omega(n)}{(\log n)^2}$. Set $\rho=r-c_m$. Then $\rho$ satisfies the refined hypothesis of Proposition~\ref{prop:weighted_permutation_sum}. Since $r=\Theta(n/\log n)$, we have $N/r^3=o(1)$. By \eqref{eq:second_moment_reduction}, \eqref{eq:chromatic_prob_link}, and Proposition~\ref{prop:N_pi_final_bound},
\begin{align*}
        \E Z^2
        &=(n!)^m\sum_{\pi\in\mathfrak S_n^m}r^{-mN}P_{G_\pi}(r)\\
        &\leq (1+o(1))(n!)^mE_{n,m,r}^2
        \sum_{\pi\in\mathfrak S_n^m}
        \exp\left(\frac{\wt(\pi)}{r-c_m}\right)\\
        &=(1+o(1))(n!)^mE_{n,m,r}^2\cdot (n!)^m S_n(r-c_m)\\
        &=(1+o(1))(\E Z)^2,
\end{align*}
where the last equality uses \eqref{eq:EZ_total} and the strict sum bound $S_n(r-c_m) = 1+o(1)$ from Proposition~\ref{prop:weighted_permutation_sum}.  Since $\E Z^2\geq(\E Z)^2$, we have $\E Z^2=(1+o(1))(\E Z)^2$. The Paley--Zygmund inequality yields 
$$\Pr(Z>0)\geq \frac{(\E Z)^2}{\E Z^2}=1-o(1).$$ 
Thus a rainbow stacking exists with high probability.
\end{proof}

\medskip

\section{Bounding the chromatic polynomial}\label{sec:bound_chromatic}

In this section, we prove Proposition~\ref{prop:N_pi_final_bound} by relying on a second-order asymptotic expansion for the chromatic polynomial. Fadnavis \cite{Fadnavis} established explicit bounds on the coefficients of the logarithmic expansion of the chromatic polynomial by using Sokal's~\cite{Sokal} and Borgs'~\cite{Borgs} bounds on its complex roots. This yields the following powerful estimate.

\begin{theorem}[Fadnavis \cite{Fadnavis}]\label{thm:fadnavis_expansion}
Fix $\Delta\geq1$. Let $G$ be a simple graph with $v$ vertices, $e$ edges, $T$ triangles, and maximum degree at most $\Delta$. Then, as $q\to\infty$,
\begin{equation}\label{eq:chromatic_formula}
        P_G(q)=q^v\exp\left( -\frac eq-\frac{e/2+T}{q^2}+O_\Delta\left(\frac v{q^3}\right) \right),
\end{equation}
uniformly over all such graphs $G$.
\end{theorem}

We apply this expansion to the auxiliary graph $G_\pi$. First, we enumerate its edges, maximum degree, and establish a sharp lower bound on its triangles.

\begin{lemma}\label{lem:aux_basic}
For every $\pi\in\mathfrak S_n^m$, the graph $G_\pi$ has exactly 
$$e(G_\pi)=m(m-1)N-\wt(\pi)$$
edges, and its maximum degree is at most $2m-2$.
\end{lemma}

\begin{proof}
Consider a fixed pair of distinct layers $k$ and $\ell$. Condition \textup{(i)} in Definition~\ref{def:aux_graph} generates a perfect matching of $N$ edges between these layers. Condition \textup{(ii)} independently generates another perfect matching of $N$ edges. 

An edge $\beta_k(e)\beta_\ell(e')$ belongs to both matchings if and only if $e=e'$ and $\pi_k(e)=\pi_\ell(e)$, which means $e$ is fixed by $\tau_{k\ell}$. As discussed, there are exactly $\wt_{k\ell}(\pi)$ such overlapping edges. By inclusion-exclusion, the total number of edges between layers $k$ and $\ell$ is $2N-\wt_{k\ell}(\pi)$. Summing over all $\binom m2$ layer pairs yields:
\[
        e(G_\pi)=\sum_{k<\ell}\bigl(2N-\wt_{k\ell}(\pi)\bigr) =m(m-1)N-\wt(\pi).
\]

For the degree bound, a fixed vertex $\beta_k(e)$ has exactly one condition-\textup{(i)} neighbour and one condition-\textup{(ii)} neighbour in any other layer $\ell$. Thus, it has at most 2 neighbours per alternative layer, yielding a maximum degree of $2(m-1) = 2m-2$.
\end{proof}

\begin{lemma}\label{lem:triangle_bound}
For every $\pi\in\mathfrak S_n^m$, if $T_\pi$ is the number of triangles in $G_\pi$, then
\begin{equation}\label{eq:triangle_lower_bound}
T_\pi\geq T_0-\frac{m-2}{3}\wt(\pi), \qquad\text{where}\qquad T_0=\frac{m(m-1)(m-2)}3N.
\end{equation}
\end{lemma}

\begin{proof}
There are two natural families of $m$-cliques in $G_\pi$. For every $e\in\binom{[n]}2$, the set
\[
        A_e=\{\beta_1(e),\ldots,\beta_m(e)\}
\]
is a clique, coming from condition \textup{(i)} in the definition of $G_\pi$. These identity cliques contain $N\binom m3$ triangle occurrences. Similarly, for every final edge $f\in\binom{[n]}2$, the set
\[
        B_f=\{\beta_1(\pi_1^{-1}(f)),\ldots,
              \beta_m(\pi_m^{-1}(f))\}
\]
is a clique, coming from condition \textup{(ii)}. These $\pi$-cliques contain another $N\binom m3$ triangle occurrences. Therefore the two families contain altogether
\[
        2N\binom m3=\frac{m(m-1)(m-2)}3N=T_0
\]
triangle occurrences.

A triangle counted in the first family is determined by an edge $e$ and a triple of distinct layers. A triangle counted in the second family is determined by a final edge $f$ and a triple of distinct layers. The only way a triangle occurrence can be counted in both families is if the same three vertices appear in both descriptions. This happens precisely when there are distinct layers $a,b,c$ and an edge $e$ such that
\[
        \pi_a(e)=\pi_b(e)=\pi_c(e).
\]
Let $D$ be the number of such double-counted triangles. Each triple equality implies the three pair equalities
\[
        \pi_a(e)=\pi_b(e),\qquad
        \pi_a(e)=\pi_c(e),\qquad
        \pi_b(e)=\pi_c(e).
\]
Thus it contributes three incidences to the pair-collision count $\wt(\pi)$. Conversely, if a pair collision $\pi_k(e)=\pi_\ell(e)$ is fixed, then there are at most $m-2$ choices of a third layer that can extend it to a triple equality. Hence
\[
        3D\leq (m-2)\wt(\pi),
        \qquad
        D\leq \frac{m-2}{3}\wt(\pi).
\]
After subtracting these possible double-counts, the graph has at least $T_0-D$ distinct triangles. This proves \eqref{eq:triangle_lower_bound}.
\end{proof}

Now gathering Theorem~\ref{thm:fadnavis_expansion}, lemmas~\ref{lem:aux_basic} and ~\ref{lem:triangle_bound} together, we are ready to prove Proposition~\ref{prop:N_pi_final_bound}.

\begin{proof}[Proof of Proposition~\ref{prop:N_pi_final_bound}]
Let $w=\wt(\pi)$. By Lemma~\ref{lem:aux_basic}, the graph $G_\pi$ has $mN$ vertices, maximum degree at most $2m-2$, and $e(G_\pi)=m(m-1)N-w.$
By Lemma~\ref{lem:triangle_bound}, we have $T_\pi\geq T_0-\frac{m-2}{3}w,$ where $T_0=\frac{m(m-1)(m-2)}3N$.
Applying Theorem~\ref{thm:fadnavis_expansion} with $q=r$ gives, uniformly in $\pi$,
\begin{align}
\log\frac{P_{G_\pi}(r)}{r^{mN}}
&=-\frac{m(m-1)N-w}{r}
  -\frac{(m(m-1)N-w)/2+T_\pi}{r^2}
  +O_m\left(\frac N{r^3}\right) \notag\\
&\leq
  -\frac{m(m-1)N}{r}
  -\frac{\frac12m(m-1)N+T_0}{r^2}
  +\frac wr+\frac{\left(\frac12+\frac{m-2}{3}\right)w}{r^2}
  +o(1). \label{eq:prop_mid}
\end{align}
Here $N/r^3=o(1)$ was used in the last line. The coefficient of $w/r^2$ is
\[
        \frac12+\frac{m-2}{3}=\frac{2m-1}{6}=c_m,
\]
and the second-order constant term is
\[
        \frac12m(m-1)N+T_0
        =\frac{m(m-1)(2m-1)}6N.
\]
On the other hand, from \eqref{eq:EZ_sigma},
\begin{align*}
2\log E_{n,m,r}
&=2N\sum_{i=1}^{m-1} \log\left(1-\frac ir\right)\\
&=-\frac{m(m-1)N}{r}
  -\frac{m(m-1)(2m-1)N}{6r^2}
  +O_m\left(\frac N{r^3}\right)\\
&=-\frac{m(m-1)N}{r}
  -\frac{\frac12m(m-1)N+T_0}{r^2}
  +o(1).
\end{align*}
Substituting this into \eqref{eq:prop_mid} yields
\[
        \log\frac{P_{G_\pi}(r)}{r^{mN}E_{n,m,r}^2}
        \leq \frac wr+\frac{c_mw}{r^2}+o(1).
\]
Since $0\leq w\leq\binom m2N$, $c_m=O_m(1)$ and $r\to\infty$, for all large $n$ we have $r>2c_m$ and
\[
\frac{w}{r-c_m}
=\frac wr+\frac{c_mw}{r^2}+O_m\left(\frac{w}{r^3}\right)
=\frac wr+\frac{c_mw}{r^2}+o(1),
\]
where the $o(1)$ is uniform over all $\pi$. Hence
$\log\frac{P_{G_\pi}(r)}{r^{mN}E_{n,m,r}^2}
 \le \frac{\wt(\pi)}{r-c_m}+o(1)$ uniformly, and hence
$$P_{G_\pi}(r)\le (1+o(1))\cdot r^{mN}E_{n,m,r}^2
\exp\left(\frac{\wt(\pi)}{r-c_m}\right)$$
as required.
\end{proof}
\medskip

\section{Refined weighted permutation sum}\label{sec:weighted_permutation_sum}

In this section, we prove Proposition~\ref{prop:weighted_permutation_sum}. We estimate the normalized sum
\begin{equation}\label{eq:core_sum}
        S_n(\rho):=\frac1{(n!)^m}
        \sum_{\pi\in\mathfrak S_n^m}\exp\left(\frac{\wt(\pi)}\rho\right).
\end{equation}
Note that $S_n(\rho)\geq 1$ since $\exp(\wt(\pi)/\rho)\geq1$ for any $\pi$ and $\rho$. Therefore it is enough to show 
$S_n(\rho)\leq 1+o(1)$ uniformly whenever \(\rho\ge R_n+\omega(n)/(\log n)^2\). 

To evaluate this sum and overcome the highly dependent pairwise overlapping constraints, we adopt a step-by-step strategy:
\begin{enumerate}
    \item \textbf{Kruskal Majorization (Section~\ref{subsec:preliminaries}):} We first decouple the complex \(\binom{m}{2}\) pairwise permutation interactions by upper-bounding the total collision weight \(\wt(\pi)\) using only the \(d = m-1\) edges of a Kruskal maximum spanning tree.
    \item \textbf{The Easy Case -- Small Weights (Section~\ref{subsec:easy_regime}):} We split the configuration space based on the maximum Kruskal edge weight \(K_1(\pi)\). If \(K_1(\pi)\) is small, all relative permutations have very few fixed points. We show this globally sparse regime naturally contributes at most \(1+o(1)\).
    \item \textbf{The Hard Case I -- Near-Identity (Section~\ref{subsec:hard_regime_near}):} When \(K_1(\pi)\) is large, we reduce the sum over permutations to a sum over abstract tree defect data (fixed points and transpositions). We first analyze the ``near-endpoint'' regime where permutations are almost identity mappings. We show that the strict Kruskal edge-weight ordering forces a strict monotonic ordering on the defects, inducing an exponential decay that strictly bounds this part by \(o(1)\).
    \item \textbf{The Hard Case II -- Non-Endpoint (Section~\ref{subsec:hard_regime_non}):} For the remaining non-identity configurations, we bound the discrete factorials via a continuous multi-variate Gamma function. Boundary optimization shows this contribution is strictly $o(1)$.
\end{enumerate}

Finally, in Section~\ref{subsec:synthesis}, we assemble these bounds to complete the proof.

\subsection{Preliminaries}\label{subsec:preliminaries}

Let
\[
        L:=\log(n!), \qquad N:=\binom n2,
        \qquad d:=m-1,
        \qquad A:=\sum_{j=1}^d j=\binom m2 .
\]
Throughout this section, $C_m$ denotes a positive constant depending only on the fixed integer $m$; its value may change from line to line. Positive constants denoted by $c_m^\ast,c_m',\kappa_m$ also depend only on $m$. These constants are unrelated to the threshold constant $c_m=(2m-1)/6$ in the main text.

For a permutation $\tau\in\mathfrak S_n$, let
\[
        q(\tau):=\left|\{e\in\tbinom{[n]}2:\tau(e)=e\}\right| .
\]
Let $f(\tau)$ be the number of fixed vertices of $\tau$, and let $t(\tau)$ be the number of $2$-cycles of $\tau$. Then
\begin{equation}\label{eq:q-ft-detailed}
        q(\tau)=\binom{f(\tau)}2+t(\tau).
\end{equation}
Indeed, a fixed two-set is either formed by two fixed vertices, or is the support of a transposition.

\begin{lemma}\label{lem:rho-normalization-detailed}
It is enough to prove Proposition~\ref{prop:weighted_permutation_sum} under the additional assumption
\begin{equation}\label{eq:rho-endpoint-detailed}
        \rho=R_n+\frac{\omega_0(n)}{(\log n)^2},
        \qquad \omega_0(n):=\min\{\omega(n),\log n\}.
\end{equation}
Under this assumption,
\begin{equation}\label{eq:rho-asymptotics-refined}
        \frac n\rho=\left(\frac4m+o(1)\right)\log n,
        \qquad
        \frac N\rho=\left(\frac2m+o(1)\right)L,
        \qquad
        \frac L\rho=\left(\frac4m+o(1)\right)(\log n)^2,
\end{equation}
and
\begin{equation}\label{eq:all-diagonal-decay}
        \exp\left(\frac{AN}{\rho}-dL\right)
        \le \exp\left(-\left(\frac{4d}{m}+o(1)\right)\omega_0(n)\right)=o(1).
\end{equation}
\end{lemma}

\begin{proof}
Because $\wt(\pi)\ge0$, the function $\rho\mapsto S_n(\rho)$ is non-increasing. Therefore an upper bound proved at the smaller value in \eqref{eq:rho-endpoint-detailed} applies to every larger admissible $\rho$. The lower bound is automatic: each summand in \eqref{eq:target_sum} is at least $1$, so $S_n(\rho)\ge1$.

Since $R_n=mN/(2L)$, $L\sim n\log n$, $N\sim n^2/2$, we have $R_n\sim mn/(4\log n)$, which gives \eqref{eq:rho-asymptotics-refined} and
\[
\frac{AN}{R_n}=\frac{m(m-1)}2\cdot \frac{2L}{m}=dL.
\]
Consequently,
\begin{align*}
        \frac{AN}{\rho}-dL
        &=AN\left(\frac1\rho-\frac1{R_n}\right)
          =-\frac{AN(\rho-R_n)}{R_n\rho}                    \\
        &=-dL\frac{\rho-R_n}{\rho}
          =-d\,\frac{L}{\rho}\cdot \frac{\omega_0(n)}{(\log n)^2}
          =-\left(\frac{4d}{m}+o(1)\right)\omega_0(n),
\end{align*}
which proves \eqref{eq:all-diagonal-decay}.
\end{proof}

Before evaluating the sum $S_n(\rho)$, we must solve the primary obstacle: the total collision weight $\wt(\pi)$ is a sum over all $\binom{m}{2}$ pairs of layers, and these pairwise overlaps are highly dependent. Direct summation over $\mathfrak{S}_n^m$ subject to these dense, clique-like constraints is very hard. To break this dependency cycle, it is crucial to decouple the interactions by reducing them to a \emph{tree structure}. We achieve this by extracting a maximum spanning tree via Kruskal's algorithm. Kruskal's bottleneck property lets us bound the $\binom{m}{2}$ pairwise interactions using just $d = m-1$ tree edges. This decoupling is essential for simplifying the sum along the tree without cyclic structure.

\begin{lemma}[Kruskal majorization]\label{lem:kruskal-majorization-detailed}
For $a,b\in[m]$, let $w_{ab}\ge0$ be edge weights on the complete graph on $[m]$. Let $K_1\ge\cdots\ge K_d$ be the edge weights selected by Kruskal's maximum spanning tree algorithm, in the order in which they are selected. Then
\begin{equation}\label{eq:kruskal-majorization-detailed}
        \sum_{1\le a<b\le m}w_{ab}\le \sum_{j=1}^d jK_j .
\end{equation}
Consequently, for every $\pi=(\pi_1,\ldots,\pi_m)\in\mathfrak S_n^m$, if Kruskal is run with weights $w_{ab}=q(\pi_a^{-1}\pi_b)$, then
\begin{equation}\label{eq:refined-tree-majorization}
        \wt(\pi)\le \sum_{j=1}^d jK_j .
\end{equation}
\end{lemma}

\begin{proof}
Let $T$ be the maximum spanning tree returned by Kruskal. The standard bottleneck property of a maximum spanning tree says that, for every pair $a,b$,
\[
        w_{ab}\le \min_{e\in P_T(a,b)}w(e),
\]
where $P_T(a,b)$ is the path from $a$ to $b$ in $T$. Indeed, if $w_{ab}$ were larger than the minimum edge on this path, replacing such a minimum edge by $ab$ would produce a spanning tree of larger total weight, contradicting maximality.

Put $K_{d+1}:=0$. For $x>0$, let $T_x$ be the forest of edges of $T$ whose weights are at least $x$, and let $c(x)$ be the number of unordered vertex pairs contained in the same component of $T_x$. Then
\[
        \sum_{a<b}\min_{e\in P_T(a,b)}w(e)=\int_0^\infty c(x)\,dx .
\]
If exactly $s$ tree edges have weight at least $x$, then $T_x$ is a forest with $s$ edges. The number of pairs inside components is maximized when these $s$ edges form one connected component on $s+1$ vertices; hence $c(x)\le \binom{s+1}{2}$. Therefore
\begin{align*}
        \sum_{a<b}w_{ab}
        &\le \int_0^\infty c(x)\,dx
         \le \sum_{s=1}^d \binom{s+1}{2}(K_s-K_{s+1})        \\
        &=\sum_{s=1}^d
          \left(\binom{s+1}{2}-\binom{s}{2}\right)K_s
          =\sum_{s=1}^d sK_s .
\end{align*}
Taking $w_{ab}=q(\pi_a^{-1}\pi_b)$ and using the definition of $\wt(\pi)$ proves \eqref{eq:refined-tree-majorization}.
\end{proof}

With these preliminary tools in hand, we proceed to evaluate the sum \(S_n(\rho)\) by partitioning the configuration space into two regimes. The split is determined by the maximum Kruskal-selected edge weight \(K_1(\pi)\). We first dispense with the structurally simple case where all relative permutations have very few fixed points.

\medskip

\subsection{The easy case: small weights}\label{subsec:easy_regime}

Define
\[
        x_0:=\lceil\log n\rceil,
        \qquad K_0:=\binom{x_0}2,
        \qquad
        y:=\left\lfloor\frac n{\log n}\right\rfloor .
\]
Let $K_1(\pi)$ be the largest Kruskal-selected edge weight given $\pi=(\pi_1,...,\pi_m)$.

\begin{lemma}\label{lem:low-weight-part-detailed}
The contribution to $S_n(\rho)$ from tuples with $K_1(\pi)\le K_0$ is at most $1+o(1)$.
\end{lemma}

\begin{proof}
The first edge chosen by Kruskal has maximum weight among all $\binom m2$ pair weights. Hence $K_1(\pi)\le K_0$ implies $q(\pi_a^{-1}\pi_b)\le K_0$ for every $a<b$, and therefore
\[
        \wt(\pi)\le AK_0=O_m((\log n)^2).
\]
By \eqref{eq:rho-asymptotics-refined}, $AK_0/\rho=O_m((\log n)^3/n)=o(1)$. Thus the normalized contribution of this part is at most
\[
\frac1{(n!)^m}\cdot (n!)^m\exp(AK_0/\rho)=1+o(1).
\qedhere\]
\end{proof}

Having shown that permutations with globally small overlaps contribute negligibly, we turn our attention to the complementary hard regime where \(K_1(\pi)>K_0\). In this regime, at least one relative permutation is close to the identity, creating dense local dependencies.

\medskip

\subsection{The hard case I: tree sum and the near-identity endpoint}\label{subsec:hard_regime_near}

To decouple the complex pairwise constraints, the following lemma formalizes the transition from a sum over permutations to a sum over abstract tree topologies and their defect data.

\begin{lemma}\label{lem:tree-data-overcount-detailed}
The contribution to $S_n(\rho)$ from tuples with $K_1(\pi)>K_0$ is at most
\begin{equation}\label{eq:exact-ft-tree-sum}
 C_m
 \sum_{\substack{(f_j,t_j)_{j=1}^d:\ K_1\ge\cdots\ge K_d,
 \ K_1>K_0}}\prod_{j=1}^d
\frac{\exp(jK_j/\rho)}{f_j!2^{t_j}t_j!},
\qquad K_j:=\binom{f_j}2+t_j,
\end{equation}
where the sum is over integer data satisfying $0\le f_j\le n$ and $0\le t_j\le (n-f_j)/2$.
\end{lemma}

\begin{proof}
We first record the required counting bound. For fixed integers $0\le f\le n$ and $0\le t\le (n-f)/2$, the number of permutations $\tau\in\mathfrak S_n$ with $f(\tau)=f$ and $t(\tau)=t$ is at most
\begin{equation}\label{eq:fixed-transposition-count-detailed}
        \binom nf\binom{n-f}{2t}\frac{(2t)!}{2^t t!}(n-f-2t)!
        =\frac{n!}{f!2^t t!}.
\end{equation}
This is an upper bound, rather than an equality, because the final factor freely permutes the remaining vertices and may introduce additional fixed points or transpositions.

Fix the rooted tree, the orientation of each selected edge away from the root, and the order in which the $d$ tree edges are selected. Since $m$ is fixed, the number of such choices is at most $C_m$. For one fixed choice, write the oriented edges as $u_jv_j$, $1\le j\le d$, and set $\tau_j:=\pi_{u_j}^{-1}\pi_{v_j}$. The root permutation $\pi_o$ together with $\tau_1,\ldots,\tau_d$ determines $\pi$ uniquely by propagation along the rooted tree. Conversely, if we sum over all possible relative permutations with prescribed $(f_j,t_j)$, we may count tuples that do not realize the specified Kruskal tree, but this only gives an upper bound.

By Lemma~\ref{lem:kruskal-majorization-detailed},
$\wt(\pi)\le \sum_{j=1}^d jK_j$. The root $\pi_o$ contributes a factor $n!$. For each $j$, \eqref{eq:fixed-transposition-count-detailed} contributes at most $n!/(f_j!2^{t_j}t_j!)$. After division by $(n!)^m=(n!)^{d+1}$, all powers of $n!$ cancel, yielding \eqref{eq:exact-ft-tree-sum}.
\end{proof}

With the sum now majorized by independent tree parameters, we further partition the parameter space. The dominant contribution theoretically stems from the ``near-identity'' endpoint, where every selected tree edge represents a permutation with almost all vertices fixed. We bound this highly structured portion first.

\begin{lemma}\label{lem:near-endpoint-contribution-detailed}
The part of the tree sum \eqref{eq:exact-ft-tree-sum} in which $f_j\ge n-y$ for every $j$ is $o(1)$.
\end{lemma}

\begin{proof}
Write $s_j:=n-f_j$, so $0\le s_j\le y$. We first identify the ordering forced by $K_1\ge\cdots\ge K_d$. If $f=n-s$, then
\begin{align*}
        \binom f2+t
        &=\binom{n-s}{2}+t
          =N-\left(\left(n-\frac12\right)s-\frac{s^2}{2}-t\right).
\end{align*}
Define $D(s,t):=\left(n-\frac12\right)s-\frac{s^2}{2}-t$ and $\mathcal D_s:=\{D(s,t):0\le t\le\lfloor s/2\rfloor\}.$ For $s\le y$, we have
\begin{align*}
\min \mathcal D_{s+1}-\max \mathcal D_s
=D\left(s+1,\left\lfloor\frac{s+1}{2}\right\rfloor\right)-D(s,0) \ge n-2-\frac{3s}{2}>0
\end{align*}
for all sufficiently large $n$. Hence larger $s$ gives smaller $K=N-D(s,t)$, and the condition $K_1\ge\cdots\ge K_d$ implies
\begin{equation}\label{eq:defect-ordering-detailed}
        0\le s_1\le s_2\le\cdots\le s_d\le y .
\end{equation}

Note by Lemma~\ref{lem:rho-normalization-detailed}, the endpoint term $s_1=\cdots=s_d=0$ and $t_1=\cdots=t_d=0$ equals
\[
        T_\star:=\prod_{j=1}^d\frac{\exp(jN/\rho)}{n!}
        =\exp\left(\frac{AN}{\rho}-dL\right)=o(1).
\]
We compare all other near-endpoint terms to it. For a fixed $j$, using $K_j=N-D(s_j,t_j)$, we have
\begin{align*}
 \frac{\exp(jK_j/\rho)}{(n-s_j)!2^{t_j}t_j!}
 \bigg/ \frac{\exp(jN/\rho)}{n!}
 &=\frac{n!}{(n-s_j)!}\,
   \frac{\exp\{-jD(s_j,t_j)/\rho\}}{2^{t_j}t_j!}                 \\
 &\le \exp\left[-\left(\frac{j(n-1/2)}\rho-\log n-C_m\right)s_j\right]
       \frac{\exp(jt_j/\rho)}{2^{t_j}t_j!}.
\end{align*}
Here we used $n!/(n-s_j)!\le n^{s_j}$ and $\frac{j s_j^2}{2\rho}\le C_m s_j $, since $s_j\le y=O(n/\log n)$ and $\rho=\Theta(n/\log n)$.

For every fixed $j$, summing over all possible $t_j$ gives
\[
        \sum_{t\ge0}\frac{\exp(jt/\rho)}{2^t t!}
        =\exp\left(\frac{e^{j/\rho}}2\right)=O_m(1).
\]
Thus the near-endpoint part is at most
\begin{equation}\label{eq:endpoint-sum-repaired}
        C_mT_\star
        \sum_{0\le s_1\le\cdots\le s_d\le y}
        \exp\left(-\sum_{j=1}^d b_js_j\right),
        \qquad
        b_j:=\frac{j(n-1/2)}\rho-\log n-C_m .
\end{equation}
Set $s_j=u_1+\cdots+u_j$ with $u_a\ge0$. Then
\[
        \sum_{j=1}^d b_js_j
        =\sum_{a=1}^d u_a\sum_{j=a}^d b_j .
\]
Using \eqref{eq:rho-asymptotics-refined}, for each fixed $a\in[d]$,
\begin{align}\label{eq:suffix-positive-repaired}
        \sum_{j=a}^d b_j
        &=\left(\frac4m\sum_{j=a}^d j-(d-a+1)+o(1)\right)\log n-O_m(1) \\
        &=\left(\frac{(m-a)(m+2a-2)}m+o(1)\right)\log n-O_m(1)
          \ge c_m'\log n . \notag
\end{align}
The last coefficient is positive for every $1\le a\le d=m-1$. Hence the sum in \eqref{eq:endpoint-sum-repaired} is bounded by
\[
        \prod_{a=1}^d\sum_{u_a\ge0}\exp(-c_m'u_a\log n)=O_m(1).
\]
Multiplying by $T_\star=o(1)$, the near-endpoint contribution is $o(1)$.
\end{proof}

\subsection{The hard case II: the non-endpoint Gamma bound}\label{subsec:hard_regime_non}

It remains to consider the part of \eqref{eq:exact-ft-tree-sum} in which at least one selected relative permutation has fewer than $n-y$ fixed vertices. Let
\[
        z:=n-y+2,
        \qquad
        h(K):=\frac{1+\sqrt{1+8K}}2,
\]
so that $\binom{h(K)}2=K$.

\begin{lemma}\label{lem:nonendpoint-gamma-reduction}
The part of \eqref{eq:exact-ft-tree-sum} in which at least one coordinate satisfies $f_j<n-y$ is bounded by
\begin{equation}\label{eq:nonendpoint-gamma-sum}
        C_m\sum_{\substack{N\ge K_1\ge\cdots\ge K_d\ge0\\ K_1>K_0,\ h(K_d)\le z}}
        \exp\left(\sum_{j=1}^d\left(
        \frac{jK_j}\rho-\log\Gamma(h(K_j)+1)\right)\right).
\end{equation}
\end{lemma}

\begin{proof}
For $0\le K\le N$, define
\[
        B(K):=\sum_{\substack{f,t:\ 0\le f\le n,\ 0\le t\le(n-f)/2\\ \binom f2+t=K}}
        \frac1{f!2^t t!}.
\]
We claim that there is an absolute constant $C$ such that
\begin{equation}\label{eq:gamma-counting-bound}
        B(K)\le \frac{C}{\Gamma(h(K)+1)}.
\end{equation}
To prove this, put $h=h(K)$ and $s=\lfloor h\rfloor$. Dropping the restriction $t\le(n-f)/2$, all possible $f$'s satisfy $0\le f\le s$. For such $f$, set
$t_f:=K-\binom f2$ and $A_f:=\frac1{f!2^{t_f}t_f!}$.

For $1\le f\le s-1$,
\begin{equation}\label{eq:Af-ratio}
        \frac{A_f}{A_{f+1}}
        =\frac{f+1}{2^f\,t_f(t_f-1)\cdots(t_f-f+1)}.
\end{equation}
Since $K\ge\binom{s}{2}$, we have
\[
        t_f\ge \binom{s}{2}-\binom f2\ge f
        \qquad (1\le f\le s-1),
\]
and hence the denominator in \eqref{eq:Af-ratio} is at least $2^f f!$. Thus $A_f\le A_{f+1}/2$ for $f\ge2$, while for $f=1$ the same conclusion holds for all large $s$; the remaining finitely many $s$'s are absorbed into the constant. Also $A_0=A_1$. Hence
\begin{equation}\label{eq:B-by-As}
        B(K)\le C A_s .
\end{equation}
Now write $h=s+\theta$ with $0\le\theta<1$, and set $u:=t_s=K-\binom{s}{2}$. Then
\[
        u=\binom{s+\theta}{2}-\binom{s}{2}
          =\frac{\theta(2s+\theta-1)}2.
\]
Using $\log\Gamma(s+\theta+1)-\log\Gamma(s+1)=\int_0^\theta\psi(s+1+v)\,dv\le \theta\log(s+2)$, where $\psi=\Gamma'/\Gamma$, and the relation $\theta\le u/(s-1/2)$, we get, for all large $s$,
\[
        \frac{\Gamma(h+1)}{s!}
        \le \exp\left(\frac{u\log(s+2)}{s-1/2}\right)
        \le 2^u .
\]
Therefore $A_s=1/(s!2^u u!)\le 1/\Gamma(h+1)$, again after adjusting the absolute constant for finitely many small $s$. Together with \eqref{eq:B-by-As}, this proves \eqref{eq:gamma-counting-bound}.

If $f\le n-y-1$, then the maximum possible value of $\binom f2+t$, subject to $t\le(n-f)/2$, is attained at $f=n-y-1$ for large $n$, and satisfies
\[
        \binom f2+t
        \le \binom{n-y-1}{2}+\frac{y+1}{2}
        \le \binom{n-y+2}{2}=\binom z2 .
\]
Thus, if at least one coordinate has $f_j<n-y$, then at least one corresponding $K_j$ is at most $\binom z2$. Since the $K_j$'s are arranged in decreasing order, $K_d\le\binom z2$, equivalently $h(K_d)\le z$. Applying \eqref{eq:gamma-counting-bound} to each coordinate of \eqref{eq:exact-ft-tree-sum} gives \eqref{eq:nonendpoint-gamma-sum}.
\end{proof}

Now that the discrete sum is majorized by a continuous multivariate Gamma function, we can apply calculus to optimize this function over its domain.

\begin{lemma}\label{lem:nonendpoint-contribution-detailed}
The part of \eqref{eq:exact-ft-tree-sum} in which at least one selected relative permutation has fewer than $n-y$ fixed vertices is $o(1)$.
\end{lemma}

\begin{proof}
By Lemma~\ref{lem:nonendpoint-gamma-reduction}, it is enough to bound the sum in \eqref{eq:nonendpoint-gamma-sum}. Write
\[
        g_j(x):=\frac{j}{\rho}\binom{x}{2}-\log\Gamma(x+1),
        \qquad
        \Phi(x_1,\ldots,x_d):=\sum_{j=1}^d g_j(x_j).
\]
The summation region is contained in the compact continuous region
\begin{equation}\label{eq:continuous-region}
        \mathcal R:=\{(x_1,\ldots,x_d):
        n\ge x_1\ge\cdots\ge x_d\ge1,
        \ x_1\ge x_0,
        \ x_d\le z\}.
\end{equation}
Here $x_j=h(K_j)$, so $K_j=\binom{x_j}{2}$.

We first show that $\Phi$ can be maximized on a very small boundary set. Suppose a maximal vector has a block of equal coordinates
\[
        x_p=x_{p+1}=\cdots=x_q=X.
\]
While the other coordinates are held fixed, the contribution of this block is
\[
        F(X)=\frac{A_0}{\rho}\binom X2-B_0\log\Gamma(X+1),
        \qquad
        A_0:=\sum_{j=p}^q j,
        \quad B_0:=q-p+1.
\]
Let $\psi=\Gamma'/\Gamma$. Then
\[
        F'(X)=\frac{A_0(2X-1)}{2\rho}-B_0\psi(X+1),
        \qquad
        F''(X)=\frac{A_0}{\rho}-B_0\psi'(X+1),
\]
and $F'''(X)=-B_0\psi''(X+1)>0$ for $X\ge1$. Thus $F''$ is increasing. Moreover,
\[
        F'(1)=\frac{A_0}{2\rho}-B_0\psi(2)<0
\]
for all large $n$, since $\rho\to\infty$ and $\psi(2)=1-\gamma>0$. Hence $F'$ has no zero at which the sign changes from positive to negative; any interior critical point is a local minimum. Therefore the maximum of $F$ on its allowed interval occurs at an endpoint. Pushing every constant block to an endpoint and iterating, we obtain a maximizer whose coordinates all lie in
\begin{equation}\label{eq:boundary-set}
        \{n,z,x_0,1\}.
\end{equation}

It remains to estimate these boundary vectors. We collect the needed asymptotics. For every fixed $j$, Stirling's formula gives
\begin{equation}\label{eq:x0-bound}
        g_j(x_0)
        \le O_m\left(\frac{x_0^2}{\rho}\right)-\log\Gamma(x_0+1)
        \le -\frac12 x_0\log x_0
        \le -c_m^\ast\log n\log\log n,
\end{equation}
whereas $g_j(1)=0$. Also, for $1\le r\le d-1=m-2$, using $\rho\ge R_n$,
\begin{equation}\label{eq:partial-n-negative}
        \sum_{j=1}^r g_j(n)
        =\frac{r(r+1)}2\frac{N}{\rho}-rL
        \le \left(\frac{r(r+1)}m-r\right)L
        =-\frac{r(m-r-1)}mL .
\end{equation}
Next let $H:=\log\Gamma(n+1)-\log\Gamma(z+1)$ and $M:=N-\binom z2$. Then
\begin{equation}\label{eq:H-M-asymptotics}
        H=(1+o(1))(n-z)\log n=(1+o(1))n,
        \qquad
        M=(1+o(1))n(n-z),
\end{equation}
since $n-z=y-2=(1+o(1))n/\log n$. Hence by \eqref{eq:rho-asymptotics-refined} we have
\begin{equation}\label{eq:M-rho-H}
\frac M\rho=\left(\frac4m+o(1)\right)H.
\end{equation}
Moreover,
\begin{equation}\label{eq:z-n-difference}
        g_j(z)-g_j(n)=H-\frac{jM}{\rho}=O_m(n).
\end{equation}

Now take a boundary vector. Because of the ordering and the constraints in \eqref{eq:continuous-region}, it has the form
\[
        \underbrace{n,\ldots,n}_{a\text{ times}},
        \underbrace{z,\ldots,z}_{b\text{ times}},
        \underbrace{x_0,\ldots,x_0}_{c\text{ times}},
        \underbrace{1,\ldots,1}_{e\text{ times}},
\]
where $a+b+c+e=d$, $a\le d-1$, and $a+b+c\ge1$. Put $r:=a+b$. If $r=0$, then $c\ge1$, and \eqref{eq:x0-bound} immediately gives
\[
\Phi\le -c_m^\ast\log n\log\log n.
\]
If $1\le r\le d-1$, then \eqref{eq:partial-n-negative}, \eqref{eq:z-n-difference}, \eqref{eq:x0-bound}, and $g_j(1)=0$ imply
\[
        \Phi
        \le -\kappa_m L+C_m n
        \le -\kappa_m' L
        \le -c_m^\ast\log n\log\log n
\]
for all large $n$. Finally suppose $r=d$. Then $c=e=0$, $b=d-a\ge1$, and
\begin{align*}
\Phi=\sum_{j=1}^d g_j(n)+\sum_{j=a+1}^d\bigl(g_j(z)-g_j(n)\bigr)    \le 0+bH-\frac{M}{\rho}\sum_{j=a+1}^d j,
\end{align*}
where $\sum_{j=1}^d g_j(n)=AN/\rho-dL\le0$ follows from $\rho\ge R_n$. Since $$\frac1b\sum_{j=a+1}^d j=\frac{a+d+1}{2}=\frac{a+m}{2},$$
by \eqref{eq:M-rho-H} we have
\[
bH-\frac{M}{\rho}\sum_{j=a+1}^d j
=-\left(1+\frac{2a}{m}+o(1)\right)bH
\le -\kappa_m n.
\]
This is again at most $-c_m^\ast\log n\log\log n$. Thus every boundary vector satisfies
\begin{equation}\label{eq:boundary-estimate-final}
        \Phi(x_1,\ldots,x_d)
        \le -c_m^\ast\log n\log\log n .
\end{equation}
By the boundary reduction, \eqref{eq:boundary-estimate-final} holds throughout $\mathcal R$.

The number of integer choices for $(K_1,\ldots,K_d)$ is at most $(N+1)^d=\exp(O_m(\log n))$. Combining this with \eqref{eq:boundary-estimate-final}, the Gamma sum in \eqref{eq:nonendpoint-gamma-sum} is at most
\[
\exp(O_m(\log n))\cdot \exp(-c_m^\ast\log n\log\log n)=o(1).
\]
Therefore the non-endpoint contribution is $o(1)$.
\end{proof}

\subsection{Synthesis of the Permutation Sum}\label{subsec:synthesis}

\begin{proof}[Proof of Proposition~\ref{prop:weighted_permutation_sum}]
By Lemma~\ref{lem:rho-normalization-detailed}, we may assume \eqref{eq:rho-endpoint-detailed}. Since $\wt(\pi)\ge0$, every summand in \eqref{eq:target_sum} is at least $1$, and hence $S_n(\rho)\ge1$.

For the upper bound, split according to $K_1(\pi)\le K_0$ or $K_1(\pi)>K_0$. The first part is at most $1+o(1)$ by Lemma~\ref{lem:low-weight-part-detailed}. For the second part, Lemma~\ref{lem:tree-data-overcount-detailed} bounds the contribution by the tree sum \eqref{eq:exact-ft-tree-sum}. The part of this tree sum with all $f_j\ge n-y$ is $o(1)$ by Lemma~\ref{lem:near-endpoint-contribution-detailed}, and the complementary part, where at least one $f_j<n-y$, is $o(1)$ by Lemma~\ref{lem:nonendpoint-contribution-detailed}. Thus
\[
        S_n(\rho)\le 1+o(1).
\]
Together with the lower bound, this gives $S_n(\rho)=1+o(1)$.
\end{proof}

\begin{remark}\label{remark:final}
The improvement over Proposition 2.2 of Alon, Defant, and Kravitz~\cite{ADK} comes from the near-endpoint estimate. A rough summation over the values of \(K_j\) near \(N\) loses a term of order \(1/\log n\) in \(\rho\). Here the Kruskal order is kept at the level of the moved-vertex defects \(s_j=n-f_j\). The strict ordering of the intervals \(\mathcal D_s\) forces \(s_1\le\cdots\le s_d\), and the suffix sums in \eqref{eq:suffix-positive-repaired} are all positive. This turns the endpoint contribution into \(T_\star\cdot O_m(1)=o(1)\), which is precisely what allows the hypothesis \(\rho\ge R_n+\omega(n)/(\log n)^2\).
\end{remark}

\medskip

\section{Concluding Remarks}\label{sec:conclusion}

This paper sharpens the threshold for rainbow stackings from the constant-order window of Alon, Defant, and Kravitz~\cite{ADK} to an $o(1)$ window.
Our theorem gives the exact rounded threshold
\[
\left\lceil \frac{m\binom{n}{2}}{2\log(n!)}+\frac{2m-1}{6}\right\rceil
\]
for a density-one set of integers $n$.  The proof identifies two sources of sharpness: the local dependence between colour conflicts, captured by the second-order chromatic-polynomial expansion of the auxiliary graph $G_\pi$, and the global dependence between layer permutations, controlled through a refined weighted-permutation estimate.

A natural direction is to allow the number of layers to grow.

\begin{question}
Determine the threshold for rainbow stackings when $m=m(n)\to\infty$.
\end{question}

In this regime, the expansion
$        \frac{m\binom{n}{2}}{2\log(n!)}+\frac{2m-1}{6}$
should no longer be treated as the canonical centre once the higher-order terms in the first moment become significant.  A more robust candidate is the exact first-moment root $r_{\mathrm{FM}}(n,m)$ defined by
\[
        (m-1)\log(n!)
        +\binom{n}{2}\sum_{i=1}^{m-1}\log\left(1-\frac{i}{r_{\mathrm{FM}}}\right)=0.
\]
The main challenge is then to prove a matching second-moment bound uniformly in $m$.  Our present method relies on the fact that typical overlaps between independent permutations contribute negligibly to the weighted permutation sum.  When $m=\Theta \big(\frac{n}{\log n}\big),$
this is no longer true: since $\mathbb E[\wt(\pi)]\sim \binom{m}{2}$ and $\frac{m\binom{n}{2}}{2\log(n!)}\sim mn/(4\log n)$, the normalized typical overlap satisfies
\[
\frac{\mathbb E[\wt(\pi)]}{\frac{m\binom{n}{2}}{2\log(n!)}}\sim \frac{2m\log n}{n}.
\]
Thus new ideas seem necessary in and beyond the range $m=\Theta \big(\frac{n}{\log n}\big)$, where the second moment begins to feel the typical, rather than merely exceptional, permutation overlaps.

\medskip

\section*{Acknowledgement}

The authors would like to acknowledge the use of an AI assistant on the technical writing and verification of the mathematical derivations. Specifically, the AI assistant provided valuable support in the detailed analysis of the Gamma function bounds presented in Lemmas~\ref{lem:nonendpoint-gamma-reduction} and \ref{lem:nonendpoint-contribution-detailed}. The core framework of this paper, in particular the application of the cluster expansion for the chromatic polynomial to eliminate the constant-order window of the threshold, was developed entirely by the authors.

H. Liu and Z. Yan were supported by the Institute for Basic Science (IBS-R029-C4).

\end{document}